\documentclass[11pt, reqno]{amsart}
\RequirePackage[OT1]{fontenc}
\usepackage[colorlinks=true, urlcolor=blue, linkcolor=blue, citecolor=blue, pdftex]{hyperref}
\usepackage{csquotes}
\usepackage[utf8]{inputenc}
\usepackage[matrix,arrow,curve]{xy}
\usepackage{amsmath,amsfonts,verbatim,afterpage,euscript,mathrsfs,amssymb}
\usepackage{array}
\usepackage{amsthm}
\usepackage{dsfont}
\usepackage{hyperref}
\usepackage{enumerate}
\usepackage[english]{babel}
\usepackage{tikz}
\usetikzlibrary{shapes}
\usepackage{titlesec}
\usepackage{algorithmicx}
\usepackage{algpseudocode}
\usepackage{algorithm}

\titleformat{\section}[block]{\normalfont\bfseries\filcenter}{\itshape\thesection}{1em}{}
\titleformat{\subsection}[block]{\normalfont\bfseries}{\itshape\thesubsection}{0.9em}{}
\titleformat{\subsubsection}[block]{\normalfont\bfseries}{\itshape\thesubsubsection}{0.8em}{}
\titleformat{\caption}[block]{\normalfont}{\itshape}{0.8em}{}

\DeclareRobustCommand{\sign}{\ensuremath \text{sign}}

\newcommand{\tc}{\tilde{c}}

\newcommand{\tq}{\tilde{q}}

\newcommand{\hq}{\hat{q}}
\newcommand{\hP}{\hat{P}}

\newcommand{\E}{\mathbb{E}}

\newcommand{\R}{\mathbb{R}}

\newcommand{\mL}{\mathcal{L}}

\newcommand{\mW}{\mathcal{W}}
\newcommand{\p}{\partial}
\newcommand{\Lip}{{\rm Lip}}

\renewcommand{\p}{\partial}

\renewcommand{\d}{{\rm d}}

\newcommand{\tP}{\tilde{P}}

\newcommand{\Tr}{{\rm Tr}}

\numberwithin{equation}{section}

\newtheorem{theoreme}{Theorem}[section]
\newtheorem{proposition}[theoreme]{Proposition}

\newtheorem{lemme}[theoreme]{Lemma}
\newtheorem{definition}[theoreme]{Definition}
\newtheorem*{remarque}{Remark}

\begin{document}

\title[Hypoelliptic SDE with singular drift]{Weak well posedness for hypoelliptic stochastic differential equation with singular drift: a sharp result}

\author{P.E. Chaudru de Raynal}
\address{UNIVERSITE SAVOIE MONT BLANC, LAMA. }
\email[]{pe.deraynal@univ-savoie.fr}

%\keywords {McKean-Vlasov processes; smoothing effect; non-linear PDE; regularization by noise} 

\begin{abstract}
In this paper, we prove weak uniqueness of hypoelliptic stochastic differential equation with Hölder drift, with Hölder exponent strictly greater than 1/3. We then extend to a weak framework the previous work \cite{chaudru_strong_2012} where strong uniqueness was proved when the Hölder exponent is strictly greater than 2/3. We also show that this result is sharp, by giving a counter example to weak uniqueness when the Hölder exponent is just below 1/3.

Our approach is based on martingale problem formulation of Stroock and Varadhan and is based on smoothing properties of the associated PDE.
\end{abstract}
\maketitle
\section{Introduction}
Let $d$ be a positive integer and $\mathcal{M}_{d}(\mathbb{R})$ be the set of $d\times d$ matrices with real coefficients. For a given positive $T$, given measurable functions $F_1,F_2,\sigma: [0,T] \times \mathbb{R}^d \times \mathbb{R}^d  \to \mathbb{R}^d \times \mathbb{R}^d \times \mathcal{M}_{d}(\mathbb{R})$ and $(W_{t}, t\geq 0)$ a standard $d$-dimensional Brownian motion defined on some filtered probability space $(\Omega, \mathcal{F}, \mathbb{P}, (\mathcal{F}_t)_{ t \geq 0})$ we consider the following $\mathbb{R}^d \times \mathbb{R}^d$ system for any $t$ in $[0,T]$:

\begin{equation}\label{systemEDO}
\left\lbrace \begin{array}{llll}
\d X^1_t = F_1(t,X^1_t,X^2_t)\d t + \sigma(t,X^1_t,X_t^2) \d B_t,\qquad & X^1_0=x_1,\\
\d X^2_t = F_2(t,X^1_t,X^2_t)\d t,\qquad & X^2_0=x_2,\\
\end{array}
\right.
\end{equation}
where $x_1$ and $x_2$ belong to $\R^d$ and where the diffusion matrix $a:=\sigma \sigma^*$ is\footnote{The notation ``$*$'' stands for the transpose.} supposed to be uniformly elliptic.

In this work, we aim at proving that this system is well-posed (\emph{i.e.} there exists a unique solution), in the weak sense, when the drift is singular. Indeed, in that case, uniqueness of the associated martingale problem from Stroock and Varadhan's theory \cite{stroock_multidimensional_1979} fails since the noise of the system degenerates. We nevertheless show that under a suitable H\"{o}lder assumption on the drift, Lipschitz condition on the diffusion matrix, and hyppoellipticity condition on the system, weak well-posedness holds for \eqref{systemEDO}. By suitable, we mean that there exists a threshold for the Hölder-continuity of the drift with respect to (w.r.t.) the degenerate argument. This Hölder-exponent is supposed to be strictly greater than $1/3$. We also show that this threshold is sharp thanks to a counter-example when the Hölder exponent is strictly less than $1/3$.\\

\textbf{Mathematical background.}  It may be a rel challenge to show well-posedness of a differential system with drift less than Lipschitz (see \cite{diperna_ordinary_1989} for a work in that direction).  The Peano example is a very good illustration of this phenomenon: for any $\alpha$ in $(0,1)$ the equation 
\begin{equation}
\d Y_t = \sign(Y_t)|Y_t|^\alpha \d t,\ Y_0=0, \quad t\in [0,T],
\end{equation} 
as an infinite number of solutions of the form $\pm c_\alpha (t-t^\star)^{1/(1-\alpha)} \mathbf{1}_{[t^\star;+\infty)},\ t^\star \in [0,T]$. Nevertheless, it has been shown that this equation is well-posed (in a strong and weak sense) as soon as it is infinitesimally perturbed by a Brownian motion. More precisely, the equation
\begin{equation}\label{eq:peano1}
\d Y_t = b(Y_t) \d t + \d B_t,\ Y_0=0, \quad t\in [0,T],
\end{equation} 
admits a unique strong solution (\emph{i.e.} there exists an almost surely unique solution adapted to the filtration generated by the Brownian motion) as soon as the function $b: \R^d \ni x \mapsto b(x)\in \R^d$ is measurable and bounded. This phenomenon is known as \emph{regularization by noise.}

Regularization by noise of systems with singular drift has been widely studied in the past few years. Since the pioneering one dimensional work of Zvonkin \cite{zvonkin_transformation_1974} and its generalization to the multidimensional setting by Veretenikov \cite{veretennikov_strong_1980} (where stochastic system with bounded drift and additive noise are handled), several authors extended the result. Krylov and Röckner \cite{krylov_strong_2005} showed that SDE with additive noise and $\mathbb{L}_p$ drift (where $p$ depends on the dimension of the system) are also well-posed and Zhang \cite{zhang_strong_2005} proved the case of multiplicative noise with uniformly elliptic and Sobolev diffusion matrix. More recently, Flandoli, Issoglio and Russo \cite{flandoli_multidimensional_2014} and Delarue and Diel \cite{delarue_rough_2015} studied the case of weak well-posedness of \eqref{eq:peano1} for a distributional drift (\emph{i.e.} in the Hölder space $C^\alpha$ where $\alpha$ is respectively greater than $-1/3$ and $-2/3$, see \cite{hairer_introduction_2015} for a definition of such a space) and Catellier and Gubinelli \cite{catellier_averaging_2012} considered systems perturbed by fractional Brownian motion. We refer to the notes of Flandoli \cite{flandoli_random_2011} for a general account on this topics.\\

In our case the setting is a bit different since the noise added in the system acts only by mean of random drift (\emph{i.e.} the system degenerates). Indeed, the archetypal example of system \eqref{systemEDO} writes
\begin{equation}\label{systemEDOex}
\d X^2_t = \big(B_t + F_2(X^2_t)\big)\d t,\qquad  X^2_0=x_2,
\end{equation}
where the function $F_2$ is supposed to be only Hölder-continuous. Thus, the system can be seen as a classical ODE whose drift is perturbed by a Brownian motion: the perturbation is then of macroscopic type. We hence consider a \emph{regularization by stochastic drift.}
 
The first work in that direction is due to Chaudru de Raynal \cite{chaudru_strong_2012} where strong well-posedness of \eqref{systemEDO} is proved when the drift is Hölder continuous with Hölder exponent w.r.t. the degenerated argument strictly greater than $2/3$ and where the system is also supposed to be hyppoelliptic. Since then, several Authors have studied the strong well-posedness of \eqref{systemEDO} with different approaches and have obtained, with weaker conditions, the same kind of threshold: in \cite{wang_degenerate_2015}, the Authors used an approach based on gradient estimates on the associated semi-group to show that the system is strongly well-posed when the drift satisfies a Hölder-Dini condition with Hölder exponent of 2/3 w.r.t. the degenerate component; in \cite{fedrizzi_regularity_2016}, the Authors used a PDE approach and obtained strong well-posedness as soon as the drift is weakly differentiable in the degenerate direction, with order of derivation of 2/3. This work then ``extends'' the results to the case of a threshold of $1/3$ for weak well-posedness and shows that the result obtained is sharp thanks to a counter-example.\\

\textbf{Strategy of proof.} Our strategy relies on the martingale problem approach of Stroock and Varadhan \cite{stroock_multidimensional_1979}. We indeed know that under our assumptions the system \eqref{systemEDO} admits at least a weak solution. We then show that this solution is unique. To do so, we investigate the regularity of the (mild) solution of the associated PDE. Namely, denoting by ${\rm Tr}(a)$ the trace of the matrix $a$, ``$\cdot$'' the standard Euclidean inner product on $\R^d$ and  $\mL$ the generator of \eqref{systemEDO}:
\begin{eqnarray}
\mL &:=& \frac{1}{2}{\rm Tr}(a(t,x_1,x_2)D^2_{x_1})  + \left[F_1(t,x_1,x_2)\right] \cdot \left[D_{x_1}\right] + \left[F_2(t,x_1,x_2)\right] \cdot \left[D_{x_2}\right],\label{gengen}
\end{eqnarray} 
we exhibit a ``good'' theory for the PDE 
\begin{equation}\label{eq:thepde}
(\p_t+\mathcal{L})u = f
\end{equation}
set on the cylinder $[0,T)\times \R^{2d}$ with terminal condition $0$ at time $T$ and where the function $f$ belongs to a certain class of functions $\mathcal{F}$. 

By ``good'', we mean that we can consider a sequence of classical solutions $(u^n)_{n\geq 0}$ and associated derivative in the non-degenerate direction $(D_{x_1} u^n)_{n\geq 0}$ along a sequence of mollified coefficients $(F_1^n,F_2^n,a^n)_{n\geq 0}$ that satisfy \emph{a priori} estimates depending only on the regularity of $(F_1,F_2,a)$. By using Arzella-Ascoli Theorem, this allows to extract a converging subsequence to the mild solution of \eqref{eq:thepde} on every compact subset of $[0,T]\times \R^{2d}$.

Hence, thanks to Itô's Formula, one can show that the quantity
$$\Big(u(t,X_t^1,X_t^2) - \int_0^t f(s,X_s^1,X_s^2)\Big)_{0\leq t \leq T},$$
is a martingale. By letting the class of function $\mathcal{F}$ be sufficiently rich, this allows us to prove uniqueness of the marginals of the weak solution of \eqref{systemEDO} and then of the law itself.\\

 Here, the crucial point is that the operator is not uniformly parabolic: the second order differentiation operator in $\mathcal{L}$ only acts in the first (and non-degenerate) direction ``$x_1$''. Therefore, we expect a loose of the regularization effect w.r.t. the degenerate component of \eqref{systemEDO}. Nevertheless, we show that the noise still regularizes, even in the degenerate direction, by mean of the random drift: we can benefit from the hypoellipticity of the system.

The system \eqref{systemEDOex} indeed relies on the so-called Kolmogorov example \cite{kolmogorov_zufallige_1934}, which is also the archetypal example of hypoelliptic system without elliptic diffusion matrix. In our setting, the hypoellipticity assumption translates as a non-degeneracy assumption on the derivative of the drift function $F_2$ w.r.t. the first component. Together with the Hölder assumption, this can be seen as a weak Hörmander condition, in reference to the work of Hörmander  \cite{hormander_hypoelliptic_1967} on degenerate operators of divergent form.

Thus, our system appears as a non-linear generalization of Kolmogorov's example. Degenerate operators of this form have been studied by many authors see \emph{e.g.} the works of Di Francesco and Polidoro \cite{di_francesco_schauder_2006}, and Delarue and Menozzi \cite{delarue_density_2010}. We also emphasize that, in \cite{menozzi_parametrix_2011}, Menozzi proved the weak well-posedness of a generalization of \eqref{systemEDO} with Lipschitz drift and Hölder diffusion matrix.

Nevertheless, to the best of our knowledge a ``good'' theory, in the sense mentioned above, for the PDE \eqref{eq:thepde} has not been exhibiting yet. We here prove the aforementioned estimates by using a first order parametrix (see \cite{friedman_partial_1964}) expansion of the operator $\mathcal{L}$ defined by \eqref{gengen}. This parametrix expansion is based on the knowledge of the related linearized and frozen version of \eqref{systemEDO} coming essentially from the previous work of Delarue and Menozzi \cite{delarue_density_2010}.\\

\textbf{Minimal setting to restore uniqueness.} Obviously, all the aforementioned works, as well as this one, lead to the question of the minimal assumption that could be done on the drift in order to restore well-posedness. Having in mind that most of these works use a PDE approach, it seems clear that the assumption on the drift relies on the regularization properties of the semi-group generated by the solution. In comparison with the previous works, the threshold of $1/3$ can be seen as the price to pay to balance the degeneracy of the system: the smoothing effect of the semi-group associated to a degenerate Gaussian process is less efficient than the one of a non-degenerate Gaussian process. We prove that our assumptions are (almost) minimal by giving a counter-example in the case where the drift $F_2$ is Hölder continuous with Hölder exponent just below $1/3$.\\

Although this example concerns our degenerate case, we feel that the method could be adapted in order to obtain the optimal threshold (for the weak well posedness) in other settings. This is the reason why we wrote it in a general form. Let us briefly explain why and expose the heuristic rule behind our counter example.

It relies on the work of Delarue and Flandoli \cite{delarue_transition_2014}. In this paper, the Peano example is investigated: namely, the system of interest is
\begin{equation}\label{eq:peano}
\d Y_t = \sign(Y_t)|Y_t|^\alpha \d t + \epsilon \d B_t,\ Y_0=0, \quad \epsilon >0, \  0<\alpha<1.
\end{equation} 
The Authors studied the zero-noise limit of the system ($\epsilon \to 0$) pathwiselly. When doing so, they put in evidence the following crucial phenomenon: in small time there is a competition between the irregularity of the drift and the fluctuations of the noise. The fluctuations of the noise allow the solution to leave the singularity while the irregularity of the drift (possibly) captures the solution in the singularity. Thus, the more singular the drift is, the more irregular the noise has to be.

This competition can be made explicit. In order to regularize the equation, the noise has to dominate the system in small time. This means that there must exists a time $0<t_{\epsilon}<1$ such that, below this instant, the noise dominates the system and push the solution far enough from the singularity, while above, the drift dominates the system and constrains the solution to fluctuate around one of the extreme solution of the deterministic Peano equation. A good way to see how the instant $t_{\epsilon}$ looks like is to compare the fluctuations of the extreme solution ($\pm t^{1/(1-\alpha)}$) with the fluctuations of the noise. Denoting by  $\gamma$ the order of the fluctuations of the noise this leads to the equation
\begin{equation*}
\epsilon t_\epsilon^{\gamma} = t_\epsilon^{1/(1-\alpha)},
\end{equation*}
which gives $t_\epsilon = \epsilon^{(1-\gamma(1-\alpha))/(1-\alpha)}$ and leads to the condition:
\begin{equation}\label{eq:threshold}
\alpha > 1-1/\gamma.
\end{equation}
The counter example, which also especially compares the fluctuations of the noise with the extreme solution, leads to the same threshold and says that weak uniqueness fails below this ceil.

Obviously, cases where $\alpha<0$ have to be considered carefully. But if we formally consider the case of a Brownian perturbation, we get $\gamma =1/2$ and so $\alpha>-1$, which is the sharp threshold exhibited in the recent work of Beck, Flandoli, Gubinelli and Maureli \cite{beck_stochastic_2014}.

In our setting, as suggested by the example \eqref{systemEDOex}, the noise added in \eqref{eq:peano} can be seen as the integral of a Brownian path, which gives $\gamma=3/2$. We deduce from equation \eqref{eq:threshold} that the threshold for the Hölder-regularity of the drift is $1/3$. We finally emphasize that this heuristic rule gives another (pathwise) interpretation for our threshold in comparison with the one obtained in the non-degenerate cases. Since the noise added in our system degenerates, the fluctuations (which are typically of order $3/2$) are not enough stronger to push the solution far enough from the singularity when the drift is too much singular (say less than $C^{1/3}$). \\

\textbf{Organization of this paper.} This paper is organized as follows. In Section \ref{sec:MR}, we give our main result: weak existence and uniqueness holds for \eqref{systemEDO}. Smoothing properties of PDE \eqref{eq:thepde} are given in Section \ref{sec:pdeandproof} as well as the proof of our main result. Then, we show in Section \ref{sec:counterexample} that our result is sharp thanks to a counter-example. Finally, the regularization properties of the PDE \eqref{eq:thepde} are proved in Section \ref{sec:pde}

\section{Notations, assumptions and main results}\label{sec:MR}
\textbf{Notations.} In order to simplify the notations, we adopt the following convention: $x, y, z,\xi,$ \emph{etc.} denote the $2d-$dimensional real variables $(x_1,x_2), (y_1,y_2), (z_1,z_2), (\xi_1,\xi_2),$ \emph{etc.}. We denote by $g(t,x) = g(t,x_1,x_2)$ any function $g$ from $[0,T] \times\R^{d} \times \R^d$ to $\R^{N},\ N\in \mathbb{N}$ evaluated at point $(t,x_1,x_2)$. Below we sometimes write $X_t = (X_t^1,X_t^2)$and, when necessary, we write $(X^{t,x}_s)_{t\leq s \leq T}$ the process defined by \eqref{systemEDO} which starts from $x$ at time $t$, \emph{i.e.} such that $X_t^{t,x}=x$.

We denote by $\mathcal{M}_d(\R)$ the set of real $d \times d$ matrices, by ``${\rm Id}$'' the identity matrix of $\mathcal{M}_d(\R)$ we denote by $\mathcal{B}$ the $2d \times d$ matrix: $B=(\rm{Id}, 0_{\R^{d}\times \R^d})^*$. We write ${\rm GL}_d(\R)$ the set of $d\times d$ invertible matrices with real coefficients. We recall that $a$ denotes the square of the diffusion matrix $\sigma$, $a :=\sigma \sigma^*$. The canonical Euclidean inner product on $\R^d$ is denoted by ``$\cdot$''. 

Subsequently, we denote by $c$, $C$, $c'$, $C'$, $c''$ \emph{etc.} a positive constant, depending only on known parameters in \textbf{(H)}, given just below, that may change from line to line and from an equation to another.

For any function from $[0,T]\times \R^d \times \R^d$, we use the notation $D$ to denote the total space derivative, we denote by $D_1$ (resp. $D_2$) the derivative with respect to the first (resp. second) $d$-dimensional space component. In the same spirit, the notation $D_{z}$ means the derivative w.r.t the variable $z$. Hence, for all integer $n$, $D^n_{z}$ is the $n^{\rm{th}}$ derivative w.r.t $z$ and for all integer $m$ the $n\times m$ cross differentiations w.r.t $z$, $y$ are denoted by $D^n_{z}D^m_{y}$. Furthermore, the partial derivative $\p/\p_t$ is denoted by $\p_t$.\\

\textbf{Assumptions} \textbf{(H).} We say that assumptions \textbf{(H)} hold if the following assumptions are satisfied.
\begin{description}
\item[(H1) Regularity of the coefficients] there exist $0<\beta_i^j<1$, $1\leq i,j\leq2$ and three positive constants $C_1, C_2, C_\sigma$ such that for all $t$ in $[0,T]$ and all $(x_1,x_2)$ and $(y_1,y_2)$ in $\mathbb{R}^d \times \mathbb{R}^d$
\begin{eqnarray*} 
&&|F_1(t,x_1,x_2) - F_1(t,y_1,y_2)| \leq C_{1} (|x_1-y_1|^{\beta^1_1} + |x_2-y_2|^{\beta^2_1})\\
&&|F_2(t,x_1,x_2) - F_2(t,y_1,y_2)| \leq C_2(|x_1-y_1|+|x_2-y_2|^{\beta^2_2})\\
&&|\sigma(t,x_1,x_2) - \sigma(t,y_1,y_2)| \leq C_\sigma (|x_1-y_1| + |x_2-y_2|).
\end{eqnarray*}
Moreover, the coefficients are supposed to be continuous w.r.t the time and the exponents $\beta_i^2,\ i=1,2$ are supposed to be strictly greater than $1/3$. Thereafter, we set $\beta_2^1=1$ for notational convenience.
\item[(H2) Uniform ellipticity of $\sigma\sigma^*$] The function $\sigma\sigma^*$ satisfies the uniform ellipticity hypothesis:
\begin{equation*}
\exists \Lambda >1,\ \forall \zeta \in \mathbb{R}^{2d},\quad \Lambda^{-1}|\zeta|^2\leq \left[\sigma \sigma^*(t,x_1,x_2)\zeta\right] \cdot \zeta  \leq \Lambda |\zeta|^2,
\end{equation*}
for all $(t,x_1,x_2) \in  [0,T] \times \mathbb{R}^d \times \mathbb{R}^d$.
\item[(H3-a) Differentiability and regularity of  $\ x_1 \mapsto F_2(.,x_1,.)$]  For all $(t,x_2) \in  [0,T]\times \R^d$, the function $F_2(t,.,x_2):\ x_1 \mapsto F_2(t,x_1,x_2)$ is continuously differentiable and there exist $0<\alpha^1<1$ and a positive constant $ \bar{C}_2$ such that, for all $(t,x_2)$ in $[0,T] \times \mathbb{R}^d$ and $x_1,y_1$ in $\mathbb{R}^d$
\begin{eqnarray*} 
&&|D_{1}F_2(t,x_1,x_2) - D_{1}F_2(t,y_1,x_2)| \leq \bar{C}_2 |x_1-y_1|^{\eta}.
\end{eqnarray*} 
\item[(H3-b) Non-degeneracy of $(D_{1}F_2)(D_{1}F_2)^*$] There exists a closed convex subset $\mathcal{E} \subset {\rm GL}_d(\R)$  (the set of $d\times d$ invertible matrices with real coefficients) such that for all $t$ in $[0,T]$ and $(x_1,x_2)$ in $\R^{2d}$ the matrix $D_{1}F_2(t,x_1,x_2)$ belongs to $\mathcal{E}$. We emphasize that this implies that 
\begin{equation*}
\exists \bar{\Lambda} >1,\ \forall \zeta \in \mathbb{R}^{2d},\quad \bar{\Lambda}^{-1}|\zeta|^2\leq \left[(D_{1}F_2)(D_{1}F_2)^*(t,x_1,x_2)\zeta\right] \cdot \zeta  \leq \bar{\Lambda} |\zeta|^2,
\end{equation*}
for all $(t,x_1,x_2) \in  [0,T] \times \mathbb{R}^d \times \mathbb{R}^d$.
\end{description}

\begin{remarque}
The convexity assumption in \textbf{(H3-b)} could seems, at first sight, 
\end{remarque}

Here is the main result of this paper.
\begin{theoreme}\label{TH:theTH}
Under \textbf{(H)}, there exists a unique weak solution to \eqref{systemEDO}.
\end{theoreme}

\section{PDE result and proof of Theorem \ref{TH:theTH}}\label{sec:pdeandproof}
Let us first begin by giving the smoothing properties of the PDE \eqref{eq:thepde}. Let $(F_1^n, F_2^n, a^n)_{n\geq 0}$ be a sequence of mollified coefficients (say infinitely differentiable with bounded derivatives of all order greater than 1) satisfying \textbf{(H)} uniformly in $n$ that converges to $(F_1, F_2, a)$ uniformly on $[0,T]\times \R^d\times \R^d$ (such an example of coefficients can be found in \cite{chaudru_strong_2012}). Let us denote by $(\mathcal{L}^n)_{n \geq 0}$ the associated sequence of regularized versions of the operator $\mathcal{L}$ defined by \eqref{gengen}. We have the following result.

\begin{theoreme}\label{TH:PDEres}
Let $\mathcal{F}$ be the set of 1-Lipschitz in space functions $f:[0,T]\times \R^d \times \R^d \to \R^d$. For each $n$, the PDE \eqref{eq:thepde} with $\mathcal{L}^n$ instead of $\mathcal{L}$ admits a unique classical solution $u^n$.

Moreover, there exist a positive $\mathcal{T}_{\ref{TH:PDEres}}$, a positive $\delta_{\ref{TH:PDEres}}$ and a positive $\nu$, depending on known parameters in \textbf{(H)} only, such that for all $T$ less tan $\mathcal{T}_{\ref{TH:PDEres}}$ the solution of the regularized PDE \eqref{eq:thepde} with source term $f$ satisfies:
\begin{equation}\label{eq:defdenormenu}
||D_{2}u^n|| +  ||D_{1}u^n||_{\infty} + ||D^2_{1}u^n||_{\infty} + ||D_{1}u^n||_{\nu} \leq CT^\delta,
\end{equation}
where 
\begin{equation}
||D_{1}u^n||_{\nu} = \sup_{t\in [0,T], x_1\in \R^d} \sup_{x_2\neq z_2} \frac{|D_{1}u^n(t,x_1,x_2)-D_{1}u^n(t,x_1,z_2)|}{|x_2 - z_2|+|x_2 - z_2|^{\beta_1^2} + |x_2 - z_2|^{\beta_2^2} + |x_2 - z_2|^\nu}.
\end{equation}
Moreover, each classical solution $u^n$ is uniformly bounded on every compact subset of $[0,T]\times \R^d \times \R^d$.
\end{theoreme}
\begin{proof}
The proof of this result is postponed to Section \ref{sec:pde}.
\end{proof}

We are now in position to prove uniqueness of the martingale problem associated to \eqref{systemEDO}. Under our assumptions, it is clear from Theorem 6.1.7 of \cite{stroock_multidimensional_1979} that the system \eqref{systemEDO} has at least one weak solution (the linear growth assumption assumed here is not a problem to do so). 

Let $f:[0,T]\times \R^d \times \R^d \to \R^d$ be some 1-Lipschitz in space function, let $u^n$ be the classical solution of the regularized version of the PDE \eqref{eq:thepde} with source term $f$ and let $(X^1,X^2)$ be a weak solution of \eqref{systemEDO} starting from $x$ at time 0. Let now suppose that $T$ is less than $\mathcal{T}_{\ref{TH:PDEres}}$ given in Theorem \ref{TH:PDEres}.  Applying Itô's Formula on $u^n(t,X_t^1,X_t^2)$ we obtain that 
\begin{eqnarray*}
u^n(t,X_t^1,X_t^2) &=& u^n(0,x_1,x_2) + \int_0^t (\p_t + \mathcal{L}) u^n(s,X_s^1,X_s^2) \d s + \int_0^t D_x u^n(s,X_s^1,X_s^2)\mathcal{B}\sigma(s,X_s^1,X_s^2) \d B_s\\
&=& u^n(0,x_1,x_2) + \int_0^t (\p_t + \mathcal{L}^n)u^n(s,X_s^1,X_s^2) \d s + \int_0^t  (\mathcal{L}-\mathcal{L}^n) u^n(s,X_s^1,X_s^2) \d s\\
&& \quad +\int_0^t D_x u^n(s,X_s^1,X_s^2)\mathcal{B}\sigma(s,X_s^1,X_s^2) \d B_s\\
&=& u^n(0,x_1,x_2) + \int_0^t f(s,X_s^1,X_s^2) \d s + \int_0^t  (\mathcal{L}-\mathcal{L}^n) u^n(s,X_s^1,X_s^2) \d s\\
&& \quad +\int_0^t D_x u^n(s,X_s^1,X_s^2)\mathcal{B}\sigma(s,X_s^1,X_s^2) \d B_s,
\end{eqnarray*} 
since $u^n$ is the solution of the regularized version of \eqref{eq:thepde} and where we recall that $\mathcal{B}$ is the $2d \times d$ matrix: $\mathcal{B}=(\rm{Id}, 0_{\R^{d}\times \R^d})^*$. 

Thanks to Theorem \ref{TH:PDEres} and Arzelà -Ascoli Theorem, we know that we can extract a subsequence of $(u^n)_{n\geq 0}$ and $(D_{x_1}u^n)_{n\geq 0}$ that converge respectively to the function $u$ and $D_{x_1}u$ uniformly on compact subset of $[0,T] \times \R^d \times \R^d$. Thus, together with the uniform convergence of the regularized coefficients, we can deduce that
\begin{equation}\label{eq:mgprop}
\left(u(t,X_t^1,X_t^2) - \int_0^t f(s,X_s^1,X_s^2) \d s - u(0,x_1,x_2) \right)_{0\leq t \leq T},
\end{equation}
is a $\mathbb{P}$-martingale by letting the regularization procedure tend to the infinity.\\

 Let us now come back to the canonical space, and let  $\mathbb{P}$ and $\tilde{\mathbb{P}}$ be two solutions of the martingale problem associated to \eqref{systemEDO} with initial condition $(x_1,x_2)$ in $\R^d\times \R^d$. Thus, for all continuous in time and Lipschitz in space functions $f :[0,T]\times \R^d \times \R^d \to \R$ we have from \eqref{eq:mgprop} (recall that $u(T,\cdot,\cdot) = 0$),
\begin{equation*}
u(0,x_1,x_2) = \E_{\mathbb{P}}\left[ \int_0^T f(s,X_s^1,X_s^2) \d s\right] = \E_{\tilde{\mathbb{P}}}\left[ \int_0^T f(s,X_s^1,X_s^2) \d s\right],
\end{equation*}
so that the marginal law of the canonical process are the same under $\mathbb{P}$ and $\tilde{\mathbb{P}}$. We extend the result on  $\R^+$ thanks to regular conditional probabilities, see  \cite{stroock_multidimensional_1979} Chapter 6.2. Uniqueness then follows from Corollary 6.2.4 of  \cite{stroock_multidimensional_1979}.

\section{Counter example} \label{sec:counterexample}
As we said, we feel that this counter example does not reduce to our current setting. Hence, we wrote it in a general form in order to be adapted to different cases. Let $\mW$ be a random process with continuous path satisfying $(\mW_t,\ t \geq 0) = (-\mW_t,\ t \geq 0) $, for all $t\geq 0$: $t^{\gamma}\mW_1 = \mW_t$ and $\E|\mW_1| < +\infty$. Let $\alpha <1$ and $c_\alpha := (1-\alpha)^{1/(1-\alpha)}$. We suppose that $\mW$ and $\alpha$ are such that there exists a weak solution of
\begin{equation}\label{eq:peano}
X_t = x + \int_0^t \rm{sign}(X_s) |X_s|^{\alpha} \d s + \mW_t,\\
\end{equation}
for any $x \geq 0$ that satisfies Kolmogorov's criterion. Given  $0<\beta<1$ we define for any continuous path $Y$ from $\R^+$ to $\R$ the variable  $\tau(Y)$ as 
$$\tau(Y) = \inf\{ t \geq 0\ : \ Y_t \leq (1-\beta)c_\alpha t^{1/(1-\alpha)}\}.$$
We now have the following Lemma:

\begin{lemme}\label{lemme:ce}
Let $X$ be a weak solution of \eqref{eq:peano} starting from some $x>0$ and suppose that $\alpha < 1-1/\gamma$. Then, there exists a positive $\rho$, depending on $\alpha$, $\beta$, $\gamma$ and $\E |\mW_1|$ only such that 
\begin{eqnarray}\label{eq:ce}
\mathbb{P}_x(\tau(X) \geq \rho) \geq 3/4.
\end{eqnarray}
\end{lemme}

We are now in position to give our counter-example. Note that if $X$ is a weak solution of \eqref{eq:peano} with the initial condition $x=0$, then, $-X$ is also a weak solution of \eqref{eq:peano}. So that, if uniqueness in law holds $X$ and $-X$ have the same law.

Let us consider a weak solution $X^n$ of \eqref{eq:peano} starting from $1/n$, $n$ being a positive integer. Since each $X^n$ satisfies Kolmogorov's criterion, the sequence of law $(\mathbb{P}_{1/n})_{n \geq 0}$ of $X^n$ is thigh, so that we can extract a converging subsequence $(\mathbb{P}_{1/n_k})_{k\geq 0}$ to $\mathbb{P}_0$, the law of the weak solution $X$ of \eqref{eq:peano} starting from 0. Since the bound in \eqref{eq:ce} does not depend on the initial condition we get that 
$$\mathbb{P}_0(\tau(X) \geq \rho) \geq 3/4,$$
and, thanks to uniqueness in law
$$\mathbb{P}_0(\tau(-X) \geq \rho) \geq 3/4,$$
which is a contradiction. Choosing $\mW = \int_0^\cdot W_s \d s$, so that $\gamma = 3/2$, we get that weak uniqueness fails as soon as 
$$\alpha < 1-1/\gamma = 1/3.$$
We now prove Lemma \ref{lemme:ce} which allows to understand how the threshold above, exhibited in the introduction, also appears in our counter-example.

\begin{proof}[Proof of Lemma \ref{lemme:ce}] Let $X$ be a weak solution of \eqref{eq:peano} starting from $x>0$. Since it has continuous path, we have almost surely that $\tau(X)>0$. Then, note that on $[0,\tau(X)]$ we have:
\begin{eqnarray*}
X_t &=& x + \int_0^t \rm{sign}(X_s) |X_s|^{\alpha} \d s + \mW_t\\
&\geq & (1-\beta)^{\alpha} c_{\alpha} t^{1/(1-\alpha)} + \mW_t.
\end{eqnarray*}
Hence, choosing $\eta$ such that $(1-\eta) = [(1-\beta)^{\alpha}+(1-\beta)]/2$ we get that:

\begin{eqnarray*}
X_t  &\geq & (1-\eta)c_\alpha t^{1/(1-\alpha)}+ (\beta-\eta) c_\alpha t^{1/(1-\alpha)} + \mW_t,
\end{eqnarray*}
for all $t$ in $[0,\tau(X)]$.

Now let $\rho$ be a positive number, set $\tc_\alpha = (\beta-\eta) c_\alpha$ and 
$$A = \left\{ \tc_\alpha t^{1/(1-\alpha)} + \mW_t > 0 \text{ for all } t \text{ in } (0,\rho]\right\}.$$ 
Note that on $A$ we have
\begin{eqnarray*}
X_t  \geq (1-\eta)c_\alpha t^{1/(1-\alpha)} \geq  (1-\beta)c_\alpha t^{1/(1-\alpha)}
\end{eqnarray*}
for all $t$ in $[0,\tau(X)]$. But this is compatible only with the event $\{\tau(X) \geq \rho\}$ so that $A \subset \{\tau(X) \geq \rho\}$. Hence
\begin{equation}
\mathbb{P}(\tau(X) \geq \rho) \geq \mathbb{P}(A).
\end{equation}

We are now going to bound from below the probability of the event $A$. We have
\begin{eqnarray*}
\mathbb{P}(A^c) &=& \mathbb{P} \left(\exists t \in (0,\rho]\ :\  \tc_\alpha t^{1/(1-\alpha)} + \mW_t \leq 0 \right)\\
&\leq & \mathbb{P} \left(\exists t \in (0,\rho]\ :\   |\mW_t | \geq  \tc_\alpha t^{1/(1-\alpha)}  \right)\\
&=& \mathbb{P} \left(\exists t \in (0,1]\ :\   (\rho t)^{\gamma} |\mW_1| \geq \tc_\alpha (\rho t)^{1/(1-\alpha)}  \right)\\
&=& \mathbb{P} \left(\exists t \in (0,1]\ :\   |\mW_1| \geq  \tc_\alpha (\rho t)^{-\delta} \right),
\end{eqnarray*}
where $\delta = \gamma - 1/(1-\alpha)$. Since $\alpha < 1-1/\gamma$, we get that $\delta>0$ and we obtain from the previous computations that 

\begin{eqnarray*}
\mathbb{P}(A^c) \leq  \mathbb{P} \left(|\mW_1| \geq  \tc_\alpha \rho^{-\delta} \right) \leq  \E|\mW_1| \tc_\alpha^{-1} \rho^{\delta},
\end{eqnarray*}
from Markov inequality. Thus
\begin{eqnarray*}
\mathbb{P}(\tau(X) \geq \rho) \geq \mathbb{P}(A) \geq  1 - \E|\mW_1| \tc_\alpha^{-1} \rho^{\delta},
\end{eqnarray*}
so that there exists a positive $\rho$ such that
\begin{eqnarray*}
\mathbb{P}(\tau(X) \geq \rho) \geq 3/4.
\end{eqnarray*}
\end{proof}

\section{Smoothing properties of the PDE}\label{sec:pde}
This section is dedicated to the proof of Theorem \ref{TH:PDEres}. This proof is in the same spirit and uses the same tools as the one used in the work \cite{chaudru_strong_2012}. In the first subsection \ref{Subsec:froesys} we recall some of the results of this work that are useful for our proof and we refer to it, especially to the sections 3 and 4, for more details. Then, we prove Theorem \ref{TH:PDEres} in Subsection \ref{Subsec:estisol}.

Theorem \ref{TH:PDEres} concerns the solution of the regularized version of \eqref{eq:thepde}. Thus, for the sake of clarity, we forget the superscript $n$ that follows from the regularization procedure and we suppose throughout this section that the coefficients $F_1,F_2$ and $a:=\sigma\sigma^*$ are smooth (say infinitely differentiable with bounded derivative of all orders greater than one).  We then specify the dependence of the constants when necessary.\\

As said in the introduction, the proof follows from a first order parametrix expansion of the solution of \eqref{eq:thepde}. This parametrix expansion (see \cite{mckean_jr._curvature_1967},\cite{friedman_partial_1964}) allows to represent the solution as a perturbation of the solution of the PDE driven by the linearized and frozen version of the operator $\mL$ defined by \eqref{gengen}. The crucial point being that we have a good knowledge of the smoothing properties of the linearized and frozen version of $\mL$. Thanks to Feynman-Kack formulae, this allows to obtain a representation of the solution in term of the semi-group associated to the linearized and frozen operator which can be estimated.

We first present the frozen and linearized frozen system and then give the smoothing properties of the associated semi-group. Then, we give the representation of the solution in term of first order parametrix expansion and we estimate it. 

\subsection{The frozen system}\label{Subsec:froesys}
Given any frozen point $(\tau,\xi)$ in $[0,T]\times \R^{2d}$, we consider the following system on $[\tau,T]$
\begin{equation}\label{thetadef}
\left\lbrace\begin{array}{ll}
\displaystyle \frac{\d}{\d s}\theta_{\tau,s}^1(\xi) = F_1(s,\theta_{\tau,s}(\xi)),\quad \theta_{\tau,\tau}^1(\xi)=\xi_1,\\
\displaystyle  \frac{\d}{\d s}\theta_{\tau,s}^2(\xi) = F_2(s,\theta_{\tau,s}(\xi)),\quad \theta_{\tau,\tau}^2(\xi)=\xi_2,
\end{array}\right.
\end{equation}
which is well posed under our regularized framework and we extend the definition of its solution on $[0,\tau)$ by assuming that for all $ (v>r)$ in $[0,T]^2$, for all $\xi$ in $\R^{2d}$, $\theta_{v,r}(\xi)= 0$. Given the solution $(\theta_{\tau,s}(\xi))_{s \leq T}$ of this system, we define the linearized and frozen version of \eqref{systemEDO}:
\begin{equation}\label{LS}
\left\lbrace\begin{array}{ll}
 \d\tilde{X}^{1,t,x}_s = F_1(s,\theta_{\tau,s}(\xi)) \d s + \sigma(s,\theta_{\tau,s}(\xi)) \d W_s\\
 \d\tilde{X}^{2,t,x}_s = \left[F_2(s,\theta_{\tau,s}(\xi)) + D_{1}F_2(s,\theta_{\tau,s}(\xi))(\tilde{X}^{1,t,x}_s-\theta^1_{\tau,s}(\xi)) \right]\d s
\end{array}\right.
\end{equation}
for all $s$ in $(t,T]$, any $t$ in $[0,T]$, and for any initial condition $x$ in $\R^{2d}$ at time $t$. We then have the following Proposition.

\begin{proposition}[Chaudru de Raynal, \cite{chaudru_strong_2012}]\label{estfundsol}
Under our assumptions:

(i) There exists a unique (strong) solution of \eqref{LS}  with mean 
$$(m^{\tau,\xi}_{t,s})_{t \leq s \leq T} =(m^{1,\tau,\xi}_{t,s},m^{2,\tau,\xi}_{t,s})_{t \leq s \leq T},  $$
where
\begin{eqnarray}\label{meanGauss}
&&m^{1,\tau,\xi}_{t,s}(x) =  x_1 + \int_t^s F_1(r,\theta_{\tau,r}(\xi)) \d r,\\
&&m^{2,\tau,\xi}_{t,s}(x) = x_2+\int_t^s  \bigg[F_2(r,\theta_{\tau,r}(\xi)) + D_{1}F_2(r,\theta_{\tau,r}(\xi))(x_1-\theta_{\tau,r}^1(\xi)) \nonumber\\
&& \hphantom{m^{2,\xi}_{\tau,s}(x)} \quad +  D_{1}F_2(r,\theta_{\tau,r}(\xi))\int_t^r F_1(v,\theta_{\tau,v}(\xi))\d v \bigg] \d r,\nonumber
\end{eqnarray} 
and uniformly non-degenerate covariance matrix $(\tilde{\Sigma}_{t,s})_{t\leq s \leq T}$:

\begin{equation}\label{covmatrice}
\tilde{\Sigma}_{t,s}=\begin{pmatrix}
\int_t^s \sigma \sigma^*(r,\theta_{\tau,r}(\xi))\d r & \int_t^s R_{r,s}(\tau,\xi) \sigma \sigma^*(r,\theta_{\tau,r}(\xi))\d r\\
 \int_t^s \sigma \sigma^*(r,\theta_{\tau,r}(\xi))R^*_{r,s}(\tau,\xi) \d r &  \int_t^s R_{t,r}(\tau,\xi)\sigma \sigma^*(r,\theta_{\tau,r}(\xi)) R_{t,r}^*(\tau,\xi) \d r\\
\end{pmatrix},					
\end{equation}
where:
\begin{equation*}
R_{t,r}(\tau,\xi)=\left[\int_t^r D_{1}F_2(v,\theta_{\tau,v}(\xi))\d v\right],\quad t\leq r\leq s\leq T.
\end{equation*}

(ii) This solution is a Gaussian process with transition density:
\begin{eqnarray}\label{gtd}
\tilde{q}(t,x_1,x_2;s,y_1,y_2)  = \frac{3^{d/2}}{(2\pi)^{d/2}}  (\det[\tilde{\Sigma}_{t,s}])^{-1/2} \exp \left( -\frac{1}{2}|\tilde{\Sigma}_{t,s}^{-1/2} (y_1-m^{1,\tau,\xi}_{t,s}(x), y_2-m^{2,\tau,\xi}_{t,s}(x))^*  |^2\right),
\end{eqnarray}
for all $s$ in $(t,T]$.

(iii) This transition density $\tilde{q}$ is the fundamental solution of the PDE driven by $\tilde{\mL}^{\tau,\xi}$ and given by:
\begin{eqnarray}
\tilde{\mL}^{\tau,\xi} &:=& \frac{1}{2} Tr\left[a(t,\theta_{\tau,t}(\xi))D^2_{x_1}\right] +  \left[ F_1(t,\theta_{\tau,t}(\xi))\right] \cdot D_{x_1}  \nonumber\\
&& \quad + \left[F_2(t,\theta_{\tau,t}(\xi))+ D_{1} F_2(t,\theta_{\tau,t}(\xi))\left(x_1 - \theta^1_{\tau,t}(\xi)\right)\right] \cdot D_{x_2}. \label{frozgen}
\end{eqnarray}

(iv) There exist two positive constants $c$ and $C$, depending only on known parameters in \textbf{(H)}, such that
\begin{equation}\label{c1defqc}
\tilde{q}(t,x_1,x_2;s,y_1,y_2) \leq C\hat{q}_{c}(t,x_1,x_2;s,y_1,y_2),
\end{equation}
where
\begin{eqnarray*}
\hat{q}_{c}(t,x_1,x_2;s,y_1,y_2)= \frac{c}{(s-t)^{2d}}\exp \left( -c \left(\frac{\big|y_1-m^{1,\tau,\xi}_{t,s}(x)\big|^2}{s-t}+ \frac{\big|y_2-m^{2,\tau,\xi}_{t,s}(x)\big|^2}{(s-t)^{3}} \right)\right),
\end{eqnarray*}
and
\begin{eqnarray}\label{estidertranker}
\left| D_{x_1}^{N^{x_1}} D_{x_2}^{N^{x_2}} D_{y_1}^{N^{y_1}} \tilde{q}(t,x_1,x_2;s,y_1,y_2)\right|  \leq C (s-t)^{-[3N^{x_2} + N^{x_1} + N^{y_1}]/2} \hat{q}_c(t,x_1,x_2;s,y_1,y_2),
\end{eqnarray}
for all $s$ in $(t,T]$ and any integers $N^{x_1},N^{x_2},N^{y_1}$ less than 2.\\
\end{proposition}

We now define some notations that appear when writing and estimating the first order parametrix expansion of the regularized solution of the PDE \eqref{eq:thepde}.
\begin{definition}\label{def:defforparam}
For all $\zeta=(\zeta_1,\zeta_2)$ in $\R^{d} \times \R^d$ we introduce the perturbation operator $\Delta(\zeta)$ as 
\begin{equation}
\Delta(\zeta) : \R^{d} \times \R^d \ni (x_1,x_2) \mapsto (x_1-\zeta_1) + (x_2-\zeta_2) \in \R^d,
\end{equation}
and for $i=1,2$
\begin{equation}
\Delta^i(\zeta): \R^{d} \times \R^d \ni (x_1,x_2) \mapsto (x_i-\zeta_i) \in \R^d.
\end{equation}
Next we set for all measurable function $\varphi : [0,T] \times \R^d \times \R^d \to \R$, for all $t<s$ in $[0,T]^2$ and $\xi$ and $x$ in $\R^{2d}$:
\begin{equation}
\left[\tP_{t,s}^\xi \varphi\right](s,x) = \int_{\R^{2d}} \varphi(s,y) \tq(t,x;s,y) \d y,
\end{equation}
and 
\begin{equation}
\left[\hP_{t,s}^\xi \varphi\right](s,x) = \int_{\R^{2d}} \varphi(s,y) \hq_c(t,x;s,y) \d y,
\end{equation}
\end{definition}

Finally, we have the following Proposition from \cite{chaudru_strong_2012} regarding the smoothing properties of $\tP$ defined above: 
\begin{proposition}\label{prop:smootheffect} Suppose that assumptions \textbf{(HR)} hold. Then, there exist three positive constants $C, C'$ and $C''$, depending on known parameters in \textbf{(H)} only such that for all $t<s$ in $[0,T]^2$, $\xi$ and $x$ in $\R^{2d}$ and all measurable function $\varphi : [0,T] \times \R^d \times \R^d \to \R$:
\begin{enumerate}[(i)]
\item $\displaystyle \Big| D_{x_i} \left[\tP_{t,s}^\xi \varphi\right](s,x) \Big| \leq C' (s-t)^{-i+1/2}   \left[\hP_{t,s}^\xi \big|\varphi\big|\right](s,x),$
\item $\displaystyle \Big| D_{x_i} \left[\tP_{t,s}^\xi \varphi\right](s,x) \Big| \leq C'' (s-t)^{-i+1/2}   \left[\hP_{t,s}^\xi \big|\varphi-\varphi(\cdot,\zeta)\big|\right](s,x),$
\end{enumerate}
for $i=1,2$ and for all $\gamma$ in $(0,1]$:
\begin{equation}\label{eq:smootheffect}
\left[\hP_{t,s}^\xi |\Delta^i(\theta_{t,s}(\xi))|^\gamma\right](s,x)\bigg|_{(\tau,\xi)=(t,x)} \leq C''' (s-t)^{(i-1/2)\gamma}.
\end{equation}
\end{proposition}
\begin{proof}
Let us recall the basics arguments of the proof, since it will be used below. Assertion (i) follows from Proposition \ref{estfundsol} and Definition \ref{def:defforparam}. Since by definition of $\tP$ we have that the quantity $\left[\tP_{t,s}^\xi \varphi(\cdot,\zeta)\right](s,x)$ does not depend on $x$ we get that $D_{x_i}\left[\tP_{t,s}^\xi \varphi(\cdot,\zeta)\right](s,x)=0$ for $i=1,2$. Then, assertion (ii) follows from the following splitting:
$$\forall  \zeta \in \R^{2d},\  \varphi = \varphi  -\varphi(\cdot,\zeta) + \varphi(\cdot,\zeta).$$
The last assertion of the Proposition follows from the Gaussian decay of $\tq$. Indeed, by definition we have
\begin{eqnarray*}
\left[\hP_{t,s}^\xi |\Delta^i(\theta_{t,s}(\xi))|^\gamma\right](s,x) &=& \int_{\R^{2d}} |y_i-\theta^i_{t,s}(\xi)|^\gamma \hq_c(t,x;s,y) \d y\\
&=& \int_{\R^{2d}} \Bigg\{ (s-t)^{(i-1/2)\gamma}\left|\frac{y_i-\theta^i_{t,s}(\xi)}{(s-t)}\right|^\gamma \frac{c}{(s-t)^{2d}}\\
&& \quad \times \exp \left( -c \left(\frac{\big|y_1-m^{1,\tau,\xi}_{t,s}(x)\big|^2}{s-t}+ \frac{\big|y_2-m^{2,\tau,\xi}_{t,s}(x)\big|^2}{(s-t)^{3}} \right)\right) \Bigg\}\d y
\end{eqnarray*}
Note that for all $s$ in $[t,T]$, the mean $(m^{1,t,x}_{t,s}(x),m^{2,t,x}_{t,s}(x))$ satisfies the ODE \eqref{meanGauss} with initial data $(t,x)$. Hence, the forward transport function defined by \eqref{meanGauss} with the initial data $(\tau,\xi) = (t,x)$ is equal to the mean: $\theta_{t,s}(x)=m^{t,x}_{t,s}(x)$. We deduce the result by letting $(\tau,\xi)=(t,x)$ and by using the following inequality:
$$ \forall \eta>0,\ \forall q>0,\ \exists \bar{C}>0 \text{ s.t. } \forall \sigma >0,\ \sigma^{q} e^{-\eta\sigma}\leq \bar{C}.$$
\end{proof}

\subsection{Estimation of the solution}\label{Subsec:estisol}
Let us now expand the regularized solution of \eqref{eq:thepde} to a a first order parametrix: we rewrite this PDE as 
$$(\p_t + \tilde{\mathcal{L}}^{\tau,\xi}) u(t,x) = - \left(\mathcal{L}-\tilde{\mathcal{L}}^{\tau,\xi}\right)u(t,x) + f(t,x) ,$$
on $[0,T)\times \R^{2d}$ with terminal condition $0$ at time $T$. Thus, using the definitions given in the previous subsection, we obtain that for every $(t,x)$ in $[0,T]\times \R^{2d}$, the solution $u$ writes
\begin{eqnarray}\label{eq:expressionu}
u(t,x) &=&  -\int_t^T\Bigg\{\left[\tP_{t,s}^\xi f\right](s,x) + \left[\tP_{t,s}^\xi (F_1-F_1(s,\theta_{t,s}(\xi)))\cdot  D_{1}u\right](s,x) \notag\\
&& + \left[\tP_{t,s}^\xi (F_2-F_2(s,\theta_{t,s}(\xi))-D_1F_2(s,\theta_{t,s}(\xi))) \cdot D_{2}u\right](s,x)\notag\\
&& + \left[\tP_{t,s}^\xi \frac{1}{2}\Tr\left[(a-a(s,\theta_{t,s}(\xi))) D^2_{1}u\right]\right](s,x) \Bigg\} \d s,
\end{eqnarray}
by choosing $\tau =t$. We made this choice for the freezing time $\tau$ in the following.\\

We next assume without loss of generality that $T<1$. We are now in position to prove the main estimates of Theorem \ref{TH:PDEres}. This is done by proving the following results and then using circular argument (see Section 4 of \cite{chaudru_strong_2012}).
\begin{proposition}
There exists four positive constants $C_1$, $C_2$, $C'$ and $C''$, and three positive numbers $\delta$, $\delta'$ and $\delta''$, depending on known parameters in \textbf{(H)} only, such that:
\begin{eqnarray}
||D^n_{1}u||_{\infty} &\leq&   C_n T^{\delta} \left(||f||_{\Lip} + ||D_1 u ||_{\infty} +  ||D_2 u ||_{\infty} \right),\quad n=1,2,\\
||D_{2}u||_{\infty} &\leq &  C' T^{\delta'} \left(||f||_{\Lip} + ||D_1 u ||_{\nu} + ||D_2 u ||_{\infty} \right),\\
||D_1 u ||_{\nu} &\leq & C'' T^{\delta''}\left(||f||_{\Lip} + ||D_1 u ||_{\nu} + ||D_2 u ||_{\infty} \right),
\end{eqnarray}
where $||\cdot||_\nu$ is defined by \eqref{eq:defdenormenu} and for all $\nu$ such that:
\begin{equation}\label{eq:conditionnu}
\nu <\inf_{i=1,2}\beta_i^2.
\end{equation}
\end{proposition}
\begin{proof} The main strategy consists in estimating the time integrands of the representation \eqref{eq:expressionu} and then to invert the differentiation and integration operators. Let $n\in \{1,2\}$ and $s$ in $(t,T]$. We have from Proposition \ref{prop:smootheffect}:
\begin{eqnarray*}
&&\Bigg| D_{x_1}^n \Bigg\{\left[\tP_{t,s}^\xi f\right](s,x) +  \left[\tP_{t,s}^\xi (F_1-F_1(s,\theta_{t,s}(\xi))) \cdot D_{1}u\right](s,x) \\
&& \quad  + \left[\tP_{t,s}^\xi (F_2-F_2(s,\theta_{t,s}(\xi))-D_1F_2(s,\theta_{t,s}(\xi))) \cdot D_{2}u\right](s,x) \\
&& \quad + \left[\tP_{t,s}^\xi \frac{1}{2}\Tr\left[(a-a(s,\theta_{t,s}(\xi))) D^2_{1}u\right]\right](s,x)\Bigg\}\Bigg|\\
&& \leq C(s-t)^{-n/2}\Bigg\{    \left[\hP_{t,s}^\xi \big|f-f(s,\theta_{t,s}(\xi))\big|\right](s,x) +    \left[\hP_{t,s}^\xi \big|(F_1-F_1(s,\theta_{t,s}(\xi))) \cdot D_{1}u\big|\right](s,x)\\
&&+ \left[\hP_{t,s}^\xi \big|\frac{1}{2}\Tr\left[(a-a(s,\theta_{t,s}(\xi))) D^2_{1}u\right]\big|\right](s,x)\\
&& \quad  +    \left[\hP_{t,s}^\xi \big|(F_2-F_2(s,\theta_{t,s}(\xi))-D_1F_2(s,\theta_{t,s}(\xi))) \cdot D_{2}u\big|\right](s,x)\Bigg\}.
\end{eqnarray*}
By using the regularity of the coefficients assumed in \textbf{(H)} (and expanding $F_2$ around the forward transport $\theta$) we get that the right hand side above is bounded by

\begin{eqnarray*}
&&  C(s-t)^{-n/2}\Bigg\{    \left[\hP_{t,s}^\xi \big|\Delta(\theta_{t,s}(\xi))(\cdot)\big|\right](s,x) \\
&& \quad +   ||D_1u||_{\infty} \left[\hP_{t,s}^\xi \Big(\big|\Delta^1(\theta_{t,s}(\xi))(\cdot)\big|^{\beta_1^1}+\big|\Delta^2(\theta_{t,s}(\xi))(\cdot)\big|^{\beta_1^2}\Big)\right](s,x)\\
&&\quad +   ||D^2_1u||_{\infty} \left[\hP_{t,s}^\xi \Big(\big|\Delta^1(\theta_{t,s}(\xi))(\cdot)\big|+\big|\Delta^2(\theta_{t,s}(\xi))(\cdot)\big|\Big)\right](s,x)\\
&&\quad +    ||D_2u||_{\infty} \left[\hP_{t,s}^\xi \Big( \big|\Delta^1(\theta_{t,s}(\xi))(\cdot)\big|^{1+\eta}+\big|\Delta^2(\theta_{t,s}(\xi))(\cdot)\big|^{\beta_2^2}\Big)\right](s,x)\Bigg\}.
\end{eqnarray*}
By letting $\xi=x$ we obtain from estimate \eqref{eq:smootheffect} in Proposition \ref{prop:smootheffect} that 

\begin{eqnarray*}
&&\Bigg| D_{x_1}^n \Bigg\{\left[\tP_{t,s}^\xi f\right](s,x) +  \left[\tP_{t,s}^\xi (F_1-F_1(s,\theta_{t,s}(\xi))) \cdot D_{1}u\right](s,x) \\
&& \quad + \left[\tP_{t,s}^\xi (F_2-F_2(s,\theta_{t,s}(\xi))-D_1F_2(s,\theta_{t,s}(\xi))) \cdot D_{2}u\right](s,x)\\
&&\quad  +\left[\tP_{t,s}^\xi \frac{1}{2}\Tr\left[(a-a(s,\theta_{t,s}(\xi))) D^2_{1}u\right]\right](s,x) \Bigg\}\Bigg|\\
&& \leq C(s-t)^{-n/2} \bigg( ||f||_{\Lip}(s-t) + ||D_1u||_{\infty} \Big( (s-t)^{\beta_1^1/2} + (s-t)^{3\beta_1^2/2} \big)\\
&&\quad + ||D^2_1u||_{\infty} \Big( (s-t)^{1/2} + (s-t)^{3/2} \big) + ||D_2u||_{\infty} \Big( (s-t)^{(1+\eta)/2}  + (s-t)^{3\beta_2^2/2} \big)\bigg),
\end{eqnarray*}
where all the time-singularities in the right hand side are integrables. Therefore

\begin{eqnarray*}
|D^n_{x_1}u(t,x)| &\leq &  C T^{(n-1)/2} \bigg(T^{}||f||_{\Lip} 
+ ||D_1 u ||_{\infty} \Big(T^{\beta_1^1/2}+T^{3\beta_1^2/2}\Big)\\
&& + ||D^2_1 u ||_{\infty} \Big(T^{1/2}+T^{3/2}\Big) +  ||D_2 u ||_{\infty} \Big(T^{(1+\eta)/2}+T^{3\beta_2^2/2}\Big)\bigg).
\end{eqnarray*}

We now estimate the derivative of the solution in the degenerate direction. By using the integration by parts argument given in Lemma 3.5 of \cite{chaudru_strong_2012}, 
\begin{eqnarray}
&&\Bigg|D_{x_2}\left[\tP_{t,s}^\xi \frac{1}{2}\Tr\left[(a-a(s,\theta_{t,s}(\xi))) D^2_{1}u\right]\right](s,x) \Bigg|\label{eq:ibp}\\
&& \leq  (s-t)^{-3/2}\Bigg\{\left[\hP_{t,s}^\xi \Big|\frac{1}{2}\Tr\left[(a-a(s,\cdot,\theta^2_{t,s}(\xi))) D^2_{1}u\right]\Big|\right](s,x)\notag\\
&&\quad  + \left[\hP_{t,s}^\xi \Big| \frac{1}{2}\Tr\left[D_1a(s,\cdot,\theta_{t,s}^2(\xi)) (D_{1}u-D_1u(s,\cdot,\theta^2_{t,s}(\xi)))\right] \Big|\right](s,x)\Bigg\}\notag\\
&& \quad + (s-t)^{-2} \left[\hP_{t,s}^\xi \Big|\frac{1}{2}\Tr\left[(a(s,\cdot,\theta_{t,s}^2(\xi)))-a(s,\theta_{t,s}(\xi))) (D_{1}u-D_1u(s,\cdot,\theta^2_{t,s}(\xi)))\right] \Big|\right](s,x).\notag
\end{eqnarray}

Thus

\begin{eqnarray*}
&&\Bigg| D_{x_2} \Bigg\{\left[\tP_{t,s}^\xi f\right](s,x) +  \left[\tP_{t,s}^\xi (F_1-F_1(s,\theta_{t,s}(\xi))) \cdot D_{1}u\right](s,x) \\
&& \quad  + \left[\tP_{t,s}^\xi (F_2-F_2(s,\theta_{t,s}(\xi))-D_1F_2(s,\theta_{t,s}(\xi))) \cdot D_{2}u\right](s,x) \\
&&\quad +\left[\tP_{t,s}^\xi \frac{1}{2}\Tr\left[(a-a(s,\theta_{t,s}(\xi))) D^2_{1}u\right]\right](s,x) \Bigg\}\Bigg|\\
&& \leq C(s-t)^{-3/2}\Bigg\{    \left[\hP_{t,s}^\xi \big|f-f(s,\cdot,\theta^2_{t,s}(\xi))\big|\right](s,x) +    \left[\hP_{t,s}^\xi \big|(F_1-F_1(s,\cdot,\theta^2_{t,s}(\xi))) D_{1}u\big|\right](s,x)\\
&&\quad +    \left[\hP_{t,s}^\xi \big|(F_1(s,\cdot,\theta^2_{t,s}(\xi)))-F_1(s,\theta_{t,s}(\xi))) (D_{1}u-D_{1}u(s,\cdot,\theta^2_{t,s}(\xi))\big|\right](s,x)\\
&& \quad+    \left[\hP_{t,s}^\xi \big|(F_2-F_2(s,\theta_{t,s}(\xi))-D_1F_2(s,\theta_{t,s}(\xi))) D_{2}u\big|\right](s,x) \\
&&\quad  +\left[\hP_{t,s}^\xi \left|\frac{1}{2}\Tr\left[(a-a(s,\cdot,\theta^2_{t,s}(\xi))) D_{1}^2u\right]\right|\right](s,x) +\left[\hP_{t,s}^\xi \Big|\frac{1}{2}\Tr\left[(a-a(s,\cdot,\theta^2_{t,s}(\xi))) D^2_{1}u\right]\Big|\right](s,x) \\
&&\quad  + \left[\hP_{t,s}^\xi \Big| \frac{1}{2}\Tr\left[D_1a(s,\cdot,\theta_{t,s}^2(\xi)) (D_{1}u-D_1u(s,\cdot,\theta^2_{t,s}(\xi)))\right] \Big|\right](s,x)\Bigg\}\\
&&\quad + (s-t)^{-2} \left[\hP_{t,s}^\xi \Big|\frac{1}{2}\Tr\left[(a(s,\cdot,\theta_{t,s}^2(\xi)))-a(s,\theta_{t,s}(\xi))) (D_{1}u-D_1u(s,\cdot,\theta^2_{t,s}(\xi)))\right] \Big|\right](s,x).
\end{eqnarray*}

By using the regularity of the coefficients assumed in \textbf{(H)} we get that the right hand side above is bounded by

\begin{eqnarray*}
&&  C(s-t)^{-3/2}\Bigg\{    \left[\hP_{t,s}^\xi \big|\Delta^2(\theta_{t,s}(\xi))(\cdot)\big|\right](s,x) +  ||D_1u||_{\nu} \left[\hP_{t,s}^\xi \Big(\big|\Delta^1(\theta_{t,s}(\xi))(\cdot)\big|^{\beta_1^1}\big|\Delta^2(\theta_{t,s}(\xi))(\cdot)\big|^{\nu}\Big)\right](s,x)\\
&& +  ||D_1u||_{\infty} \left[\hP_{t,s}^\xi \Big(\big|\Delta^2(\theta_{t,s}(\xi))(\cdot)\big|^{\beta_1^2}\Big)\right](s,x)\\
&&  +    ||D_2u||_{\infty} \left[\hP_{t,s}^\xi \Big( \big|\Delta^1(\theta_{t,s}(\xi))(\cdot)\big|^{1+\eta}+\big|\Delta^2(\theta_{t,s}(\xi))(\cdot)\big|^{\beta_2^2}\Big)\right](s,x)\\
&& +||D^2_1u||_{\infty} \left[\hP_{t,s}^\xi \Big(\big|\Delta^2(\theta_{t,s}(\xi))(\cdot)\big|\Big)\right](s,x)  + ||D_1u||_{\nu}\bigg( \left[\hP_{t,s}^\xi \Big(\big|\Delta^2(\theta_{t,s}(\xi))(\cdot)\big|^{\nu}\Big)\right](s,x)\\
&&\quad  + (s-t)^{-1/2}\left[\hP_{t,s}^\xi \Big(\big|\Delta^1(\theta_{t,s}(\xi))(\cdot)\big|\big|\Delta^2(\theta_{t,s}(\xi))(\cdot)\big|^\nu\Big)\right](s,x) \bigg)\Bigg\}.
\end{eqnarray*}
By letting $\xi=x$ we obtain from estimate \eqref{eq:smootheffect} in Proposition \ref{prop:smootheffect} that 

\begin{eqnarray*}
&&\Bigg| D_{x_2} \Bigg\{\left[\tP_{t,s}^\xi f\right](s,x) +  \left[\tP_{t,s}^\xi (F_1-F_1(s,\theta_{t,s}(\xi))) \cdot D_{1}u\right](s,x) \\
&& \quad  + \left[\tP_{t,s}^\xi (F_2-F_2(s,\theta_{t,s}(\xi))-D_1F_2(s,\theta_{t,s}(\xi))) \cdot D_{2}u\right](s,x) \\
&&\quad +\left[\tP_{t,s}^\xi \frac{1}{2}\Tr\left[(a-a(s,\theta_{t,s}(\xi))) D^2_{1}u\right]\right](s,x) \Bigg\}\Bigg|\\
&& \leq C(s-t)^{-3/2} \bigg( ||f||_{\Lip}(s-t)^{3/2} + ||D_1u||_{\nu}  (s-t)^{\beta_1^1/2+3\nu/2} +||D_1u||_{\infty}  (s-t)^{3\beta_1^2/2} \\
&&\quad + ||D_2u||_{\infty} \Big( (s-t)^{(1+\eta)/2}  + (s-t)^{3\beta_2^2/2} \Big) + ||D_1u||_{\nu}  (s-t)^{1/2+3\nu/2} +||D_1^2u||_{\infty}  (s-t)^{3/2} \\
&& \qquad + ||D_1u||_{\nu} (s-t)^{3\nu/2} \bigg).
\end{eqnarray*}
Since $\beta_j^2>1/3$, and $\nu$ is constrained by \eqref{eq:conditionnu}, all the time-singularities of the right hand side above are integrables on $(t,T]$. Hence, we deduce from  \eqref{eq:expressionu} and the estimate above that there exists a positive $\delta'$, depending on known parameters in \textbf{(H)} only, such that:  

\begin{eqnarray*}
|D_{x_2}u(t,x)| &\leq &  C T^{\delta'} \bigg(||f||_{\Lip} + ||D_1 u ||_{\nu} + ||D_1u||_{\infty} +  ||D_2 u ||_{\infty} + ||D^2_1 u ||_{\infty} + ||D^2_1u||_{\infty}\bigg).
\end{eqnarray*}

Finally, we compute the Hölder semi norm of $D_{x_1}u$. Let $x_2\neq z_2$ belong to $\R^{d}$. We have from \eqref{eq:expressionu}:
\begin{eqnarray}
&&D_{x_1}u(t,x_1,x_2)-D_{x_1}u(t,x_1,z_2) \label{eq:holdedu}\\
&&= - D^n_{x_1}\int_t^T\Bigg\{\left[\tP_{t,s}^\xi f\right](s,x_1,x_2)-\left[\tP_{t,s}^\xi f\right](s,x_1,z_2) \notag\\
&&+ \left[\tP_{t,s}^\xi (F_1-F_1(s,\theta_{t,s}(\xi))) \cdot D_{1}u\right](s,x_1,x_2)-\left[\tP_{t,s}^\xi (F_1-F_1(s,\theta_{t,s}(\xi))) \cdot D_{1}u\right](s,x_1,z_2) \notag\\
&& + \left[\tP_{t,s}^\xi (F_2-F_2(s,\theta_{t,s}(\xi))-D_1F_2(s,\theta_{t,s}(\xi))) \cdot D_{2}u\right](s,x_1,x_2)\notag\\
&&\quad -\left[\tP_{t,s}^\xi (F_2-F_2(s,\theta_{t,s}(\xi))-D_1F_2(s,\theta_{t,s}(\xi))) \cdot D_{2}u\right](s,x_1,z_2) \notag\\
&& + \left[\tP_{t,s}^\xi \frac{1}{2}\Tr\left[(a-a(s,\theta_{t,s}(\xi))) D^2_{1}u\right]\right](s,x_1,x_2)- \left[\tP_{t,s}^\xi \frac{1}{2}\Tr\left[(a-a(s,\theta_{t,s}(\xi))) D^2_{1}u\right]\right](s,x_1,z_2)\Bigg\}\d s.\notag
\end{eqnarray}

We first estimate for any $s$ in $(t,T]$ the quantity:
\begin{eqnarray}\label{eq:targetholder}
&&\Bigg| D_{x_1}\left[\tP_{t,s}^\xi f\right](s,x_1,x_2)-D_{x_1}\left[\tP_{t,s}^\xi f\right](s,x_1,z_2) \\
&&+ D_{x_1}\left[\tP_{t,s}^\xi (F_1-F_1(s,\theta_{t,s}(\xi)))\cdot D_{1}u\right](s,x_1,x_2)-D_{x_1}\left[\tP_{t,s}^\xi (F_1-F_1(s,\theta_{t,s}(\xi))) \cdot D_{1}u\right](s,x_1,z_2)\notag \\
&& + D_{x_1}\left[\tP_{t,s}^\xi (F_2-F_2(s,\theta_{t,s}(\xi))-D_1F_2(s,\theta_{t,s}(\xi))) \cdot D_{2}u\right](s,x_1,x_2)\notag\\
&& -D_{x_1}\left[\tP_{t,s}^\xi (F_2-F_2(s,\theta_{t,s}(\xi))-D_1F_2(s,\theta_{t,s}(\xi))) \cdot D_{2}u\right](s,x_1,z_2) \notag \\
&& + D_{x_1}\left[\tP_{t,s}^\xi \frac{1}{2}\Tr\left[(a-a(s,\theta_{t,s}(\xi))) D^2_{1}u\right]\right](s,x_1,x_2)\notag\\
&& - D_{x_1}\left[\tP_{t,s}^\xi \frac{1}{2}\Tr\left[(a-a(s,\theta_{t,s}(\xi))) D^2_{1}u\right]\right](s,x_1,z_2)\Bigg|.\notag
\end{eqnarray}

To do this, we split the time interval w.r.t. the characteristic time-scale of the second space variable: let $\mathcal{S}:=\left\{s\in (t,T]:\ |x_2-z_2|\leq (s-t)^{3/2} \right\}$. Note that on $\mathcal{S}$ we have for any measurable function $\varphi:[0,T]\times \R^d \times \R^d \to \R$: 
\begin{eqnarray*}
&&\left|D_{x_1}\left[\tP_{t,s}^\xi \varphi \right](s,x_1,x_2)-D_{x_1}\left[\tP_{t,s}^\xi \varphi \right](s,x_1,z_2)\right|\\ 
&& \quad \leq \sup_{\lambda \in (0,1)} \left| D_{x_2}D_{x_1}  \left[\tP_{t,s}^\xi \varphi \right](s,x_1,\lambda x_2+ (1-\lambda)z_2)\right||x_2-z_2|\\
&& \quad \leq C (s-t)^{-2} \left[\hP_{t,s}^\xi |\varphi|\right](s,x_1,x_2)|x_2-z_2|\\
&& \quad \leq C (s-t)^{-1/2-3\nu/2} \left[\hP_{t,s}^\xi |\varphi|\right](s,x_1,x_2)|x_2-z_2|^{\nu},
\end{eqnarray*}
for every $0<\nu<1$. Hence, by using this estimate together with Proposition \ref{prop:smootheffect}, by repeating the computations done when estimating $D_{x_1}u$, we have that the quantity \eqref{eq:targetholder} is bounded on $\mathcal{S}$ by

\begin{eqnarray}\label{eq:estigolder1}
&&  C(s-t)^{-1/2}\Bigg\{    \left[\hP_{t,s}^\xi \big|\Delta(\theta_{t,s}(\xi))(\cdot)\big|\right](s,x) +   ||D_1u||_{\infty} \bigg[\hP_{t,s}^\xi \Big(\big|\Delta^1(\theta_{t,s}(\xi))(\cdot)\big|^{\beta_1^1}\\
&&+\big|\Delta^2(\theta_{t,s}(\xi))(\cdot)\big|^{\beta_1^2}\Big)\bigg](s,x_1,x_2) +    ||D_2u||_{\infty} \left[\hP_{t,s}^\xi \Big( \big|\Delta^1(\theta_{t,s}(\xi))(\cdot)\big|^{1+\eta}+\big|\Delta^2(\theta_{t,s}(\xi))(\cdot)\big|^{\beta_2^2}\Big)\right](s,x_1,x_2)\notag\\
&& +   ||D^2_1u||_{\infty} \left[\hP_{t,s}^\xi \Big(\big|\Delta^1(\theta_{t,s}(\xi))(\cdot)\big|+\big|\Delta^2(\theta_{t,s}(\xi))(\cdot)\big|\Big)\right](s,x)\Bigg\}.\notag
\end{eqnarray}
Thus, by choosing $\xi=x$ we obtain that \eqref{eq:targetholder} is bounded on $\mathcal{S}$ by

\begin{eqnarray}\label{eq:holdedu1}
&&C'(s-t)^{-1/2-3\nu/2} \bigg( ||f||_{\Lip}(s-t)^{1/2} + ||D_1u||_{\infty} \Big( (s-t)^{\beta_1^1/2} + (s-t)^{3\beta_1^2/2} \big)\\
&&\quad + ||D^2_1u||_{\infty} \big( (s-t)^{1/2} + (s-t)^{3/2} \big) + ||D_2u||_{\infty} \big( (s-t)^{(1+\eta)/2}  + (s-t)^{3\beta_2^2/2} \big)\bigg),\notag
\end{eqnarray}
for all $\nu$ satisfying \eqref{eq:conditionnu}.

We now estimate \eqref{eq:targetholder} on $\mathcal{S}^c$. On a first hand, we have from the computations done when estimating $D_{x_1}u$ that:

\begin{eqnarray*}
&&\Bigg| D_{x_1} \Bigg\{\left[\tP_{t,s}^\xi f\right](s,x) +  \left[\tP_{t,s}^\xi (F_1-F_1(s,\theta_{t,s}(\xi))) \cdot D_{1}u\right](s,x) \\
&& \quad  + \left[\tP_{t,s}^\xi (F_2-F_2(s,\theta_{t,s}(\xi))-D_1F_2(s,\theta_{t,s}(\xi))) \cdot D_{2}u\right](s,x) \\
&&\quad + \left[\tP_{t,s}^\xi \frac{1}{2}\Tr\left[(a-a(s,\theta_{t,s}(\xi))) D^2_{1}u\right]\right](s,x)\Bigg\}\Bigg|\\
&& \leq   C(s-t)^{-1/2}\Bigg\{    ||f||_{\Lip}\left[\hP_{t,s}^\xi \big|\Delta(\theta_{t,s}(\xi))(\cdot)\big|\right](s,x) \\
&&+   ||D_1u||_{\infty} \left[\hP_{t,s}^\xi \Big(\big|\Delta^1(\theta_{t,s}(\xi))(\cdot)\big|^{\beta_1^1}+\big|\Delta^2(\theta_{t,s}(\xi))(\cdot)\big|^{\beta_1^2}\Big)\right](s,x)\\
&&+    ||D_2u||_{\infty} \left[\hP_{t,s}^\xi \Big( \big|\Delta^1(\theta_{t,s}(\xi))(\cdot)\big|^{1+\eta}+\big|\Delta^2(\theta_{t,s}(\xi))(\cdot)\big|^{\beta_2^2}\Big)\right](s,x)\\
&&+   ||D^2_1u||_{\infty} \left[\hP_{t,s}^\xi \Big(\big|\Delta^1(\theta_{t,s}(\xi))(\cdot)\big|+\big|\Delta^2(\theta_{t,s}(\xi))(\cdot)\big|\Big)\right](s,x)\Bigg\}.
\end{eqnarray*}
Since on $\mathcal{S}^c$ we have $1\leq (s-t)^{-3\nu/2} |x_2-z_2|^{\nu}$, by choosing $\xi=x$ and then using Proposition \ref{prop:smootheffect} it comes that

\begin{eqnarray}\label{eq:estiholder2}
&&\Bigg| D_{x_1} \Bigg\{\left[\tP_{t,s}^\xi f\right](s,x) +  \left[\tP_{t,s}^\xi (F_1-F_1(s,\theta_{t,s}(\xi))) \cdot D_{1}u\right](s,x) \\
&& \quad  + \left[\tP_{t,s}^\xi (F_2-F_2(s,\theta_{t,s}(\xi))-D_1F_2(s,\theta_{t,s}(\xi))) \cdot D_{2}u\right](s,x)\Bigg|\notag\\
&&\quad + \left[\tP_{t,s}^\xi \frac{1}{2}\Tr\left[(a-a(s,\theta_{t,s}(\xi))) \cdot D^2_{1}u\right]\right](s,x)\Bigg\}\Bigg|\notag\\
&& \leq  C(s-t)^{-1/2-3\nu/2}\Bigg\{    ||f||_{\Lip}\left[\hP_{t,s}^\xi \big|\Delta(\theta_{t,s}(\xi))(\cdot)\big|\right](s,x) +   ||D_1u||_{\infty} \bigg[\hP_{t,s}^\xi \Big(\big|\Delta^1(\theta_{t,s}(\xi))(\cdot)\big|^{\beta_1^1}\notag\\
&&+\big|\Delta^2(\theta_{t,s}(\xi))(\cdot)\big|^{\beta_1^2}\Big)\bigg](s,x_1,x_2) +    ||D_2u||_{\infty} \left[\hP_{t,s}^\xi \Big( \big|\Delta^1(\theta_{t,s}(\xi))(\cdot)\big|^{1+\eta}+\big|\Delta^2(\theta_{t,s}(\xi))(\cdot)\big|^{\beta_2^2}\Big)\right](s,x_1,x_2)\notag\\
&&+||D^2_1u||_{\infty} \left[\hP_{t,s}^\xi \Big(\big|\Delta^1(\theta_{t,s}(\xi))(\cdot)\big|+\big|\Delta^2(\theta_{t,s}(\xi))(\cdot)\big|\Big)\right](s,x) \Bigg\}|x_2-z_2|^{\nu}.\notag
\end{eqnarray}
 We emphasize that all the time singularity above are again integrables provided $\nu$ satisfies \eqref{eq:conditionnu}. It thus only remains to estimate the last part of \eqref{eq:targetholder} on $\mathcal{S}^c$, namely

\begin{eqnarray*}
&&\Bigg| D_{x_1} \Bigg\{\left[\tP_{t,s}^\xi f\right](s,x_1,z_2) +  \left[\tP_{t,s}^\xi (F_1-F_1(s,\theta_{t,s}(\xi))) \cdot D_{1}u\right](s,x_1,z_2) \\
&& \quad  + \left[\tP_{t,s}^\xi (F_2-F_2(s,\theta_{t,s}(\xi))-D_1F_2(s,\theta_{t,s}(\xi))) \cdot D_{2}u\right](s,x_1,z_2)\\
&&\qquad +\left[\tP_{t,s}^\xi \frac{1}{2}\Tr\left[(a-a(s,\theta_{t,s}(\xi))) D^2_{1}u\right]\right](s,x_1,z_2) \Bigg\}\Bigg|.
\end{eqnarray*}
The main issue here is that the estimate of Proposition \ref{prop:smootheffect} can not be applied immediately, the semi-group being evaluating at point $(s,x_1,z_2)$ and the freezing point being previously chosen as $\xi=(x_1,x_2)$. The main idea consists in re-centering all the terms above and taking advantage on the fact that $|\theta_{t,s}^2(x) - m^{2,x}_{t,s}(x_1,z_2)| \leq |x_2-z_2|$.\\

Let us first begin with the term $D_{x_1}\left[\tP_{t,s}^\xi (F_1-F_1(s,\theta_{t,s}(\xi))) \cdot D_{1}u\right](s,x_1,z_2)$. Splitting first $(F_1-F_1(s,\theta_{t,s}(\xi))) \cdot D_{1}u$ as 
$$\Big(F_1-F_1(s,\theta^1_{t,s}(\xi),m^{2,x}_{t,s}(x_1,z_2))\Big) \cdot D_{1}u + \Big(F_1(s,\theta^1_{t,s}(\xi),m^{2,x}_{t,s}(x_1,z_2))-F_1(s,\theta_{t,s}^1(\xi),\theta_{t,s}^2(\xi))\Big) \cdot D_{1}u,$$
we get
\begin{eqnarray*}
&&\left|D_{x_1}\left[\tP_{t,s}^\xi (F_1-F_1(s,\theta_{t,s}(\xi))) \cdot D_{1}u\right](s,x_1,z_2)\right| \\
&&\leq C(s-t)^{-1/2}||D_{1}u||_\infty \bigg[\hP_{t,s}^\xi \Big(|\Delta^1(\theta_{t,s}(\xi))|^{\beta_1^1} + |\Delta^2(m^{2,x}_{t,s}(x_1,z_2))|^{\beta_1^2}\\
&&\qquad  + |\theta_{t,s}^2(\xi) -m^{2,x}_{t,s}(x_1,z_2) |^{\beta_1^2}\Big) \bigg](s,x_1,z_2)\\
&& \leq C'(s-t)^{-1/2} ||D_1u||_{\infty} \Big( (s-t)^{\beta_1^1/2} + (s-t)^{3\beta_1^2/2} + |x_2-z_2|^{\beta_1^2} \big).
\end{eqnarray*}

Next we split $(F_2-F_2(s,\theta_{t,s}(\xi))-D_1F_2(s,\theta_{t,s}(\xi))) \cdot D_{2}u$ as 

$$\Big(F_2-F_2(s,\cdot,m^{2,x}_{t,s}(x_1,z_2))\Big) \cdot D_{2}u +  \Big(F_2(s,\cdot,m^{2,x}_{t,s}(x_1,z_2)) -    F_2(s,\theta_{t,s}(\xi))-D_1F_2(s,\theta_{t,s}(\xi))\Big) \cdot D_{2}u,$$
and we obtain

\begin{eqnarray*}
&&\left|D_{x_1}\left[\tP_{t,s}^\xi (F_2-F_2(s,\theta_{t,s}(\xi))-D_1F_2(s,\theta_{t,s}(\xi))) \cdot D_{2}u\right](s,x_1,z_2)\right| \\
&&\leq C(s-t)^{-1/2}||D_{2}u||_\infty \bigg[\hP_{t,s}^\xi \Big( |\Delta^2(m^{2,x}_{t,s}(x_1,z_2))|^{\beta_2^2} \\
&&\qquad + |\Delta^1(\theta_{t,s}(\xi))|^{1+\eta}  + |\theta_{t,s}^2(\xi) -m^{2,x}_{t,s}(x_1,z_2) |^{\beta_2^2}\Big) \bigg](s,x_1,z_2)\\
&& \leq C'(s-t)^{-1/2} ||D_2u||_{\infty} \Big( (s-t)^{3\beta_2^2/2} + (s-t)^{(1+\eta)/2} + |x_2-z_2|^{\beta_2^2} \big).
\end{eqnarray*}

Finally, we write $(1/2)\Tr\left[(a - a(s,\theta_{t,s}(\xi))D^2u_1\right]$ as 
$$\frac{1}{2}\Tr\left[(a-a(s,\theta^1_{t,s}(\xi),m^{2,x}_{t,s}(x_1,z_2))\Big) D_{1}^2u\right] + \frac{1}{2}\Tr\left[\Big(a(s,\theta_{t,s}^1(\xi),\theta_{t,s}^2(\xi))-a(s,\theta^1_{t,s}(\xi),m^{2,x}_{t,s}(x_1,z_2))\Big) D_{1}u\right],$$
and we obtain

\begin{eqnarray*}
&&\left|D_{x_1}\left[\tP_{t,s}^\xi\frac{1}{2}\Tr\left[(a - a(s,\theta_{t,s}(\xi))D^2u_1\right]\right](s,x_1,z_2)\right| \\
&&\leq C(s-t)^{-1/2}||D^2_{1}u||_\infty \bigg[\hP_{t,s}^\xi \Big(|\Delta^1(\theta_{t,s}(\xi))|\\
&&\qquad + |\Delta^2(m^{2,x}_{t,s}(x_1,z_2))|+ |\theta_{t,s}^2(\xi) -m^{2,x}_{t,s}(x_1,z_2) |\Big) \bigg](s,x_1,z_2)\\
&& \leq C'(s-t)^{-1/2} ||D^2_1u||_{\infty} \Big( (s-t)^{1/2} + (s-t)^{3/2} + |x_2-z_2| \Big).
\end{eqnarray*}

Hence, putting the previous estimates together, letting $\xi=x$ we get that on $\mathcal{S}^c$
\begin{eqnarray}\label{eq:estiholder3}
&&\Bigg| D_{x_1} \Bigg\{\left[\tP_{t,s}^\xi f\right](s,x_1,z_2) +  \left[\tP_{t,s}^\xi (F_1-F_1(s,\theta_{t,s}(\xi))) \cdot D_{1}u\right](s,x_1,z_2) \\
&& \quad  + \left[\tP_{t,s}^\xi (F_2-F_2(s,\theta_{t,s}(\xi))-D_1F_2(s,\theta_{t,s}(\xi))) \cdot D_{2}u\right](s,x_1,z_2)\notag\\
&& \quad +\left[\tP_{t,s}^\xi \frac{1}{2}\Tr\left[(a-a(s,\theta_{t,s}(\xi))) D^2_{1}u\right]\right](s,x_1,z_2)  \Bigg\}\Bigg|\notag\\
&&\leq C(s-t)^{-1/2} \Bigg\{ ||f||_{\Lip} \big((s-t)^{(1-3\nu)/2} +   (s-t)^{3(1-\nu)/2}\big) \notag\\
&&\quad + ||D_1u||_{\infty} \bigg( (s-t)^{(\beta_1^1-3\nu)/2} + (s-t)^{3(\beta_1^2-\nu)/2} + [x_2-z_2|^{\beta_1^2-\nu}\bigg)\notag\\
&&\quad +  ||D_2u||_{\infty}\bigg( (s-t)^{3(\beta_2^2-\nu)/2} + (s-t)^{(1+\eta-\nu)/2} + |x_2-z_2|^{\beta_2^2-\nu}\bigg)\Bigg\} |x_2-z_2|^\nu\notag\\
&&\quad + ||D^2_1u||_{\infty} \bigg( (s-t)^{(1-3\nu)/2} + (s-t)^{3(1-\nu)/2} + [x_2-z_2|^{1-\nu}\bigg),\notag
\end{eqnarray}
since $1\leq (s-t)^{-3\nu/2} |x_2-z_2|^\nu$ and this holds for all $\nu$ satisfying \eqref{eq:conditionnu}. Putting together estimates \eqref{eq:holdedu1}, \eqref{eq:estiholder2} and \eqref{eq:estiholder3}, we can invert the differentiation and integration operators in \eqref{eq:holdedu} and we deduce that there exists a positive $\delta''$ depending on known parameters in \textbf{(H)} only such that
$$||D_1u||_{\nu} \leq CT^{\delta''}(||D_1u||_{\infty}  + ||D_2u||_{\infty} + ||D^2_1u||_{\infty} ),$$
where $||\cdot||_{\nu}$ is defined in Theorem \ref{TH:PDEres}. 
\end{proof}

\textbf{Acknowledgment } I would like to thanks François Delarue and Mario Maurelli for valuable contribution to this work.

%\section{Remarque}
%La chaîne n'est pas faisable avec le parametrix: on ne peut gagner qu'un $t^{1/2+\epsilon}$ de regularité Hölder pour $D_{x_1}u$ par rapport à $x_3$. Si on cherche à estimer $D_{x_3}$, on fait sortir une singularité de $t^{3-1/2}=t^{-5/2}$ qui n'est compensée que par la dérive $t^{\beta_1^1/2}$ et de la régularité Hölder de $D_{x_1}u$: $t^{1/2+\epsilon}$. Il manque donc $1$... Par contre on peut traiter le cas où la chaîne est de ce type:
%\begin{equation}\label{systemEDO2}
%\left\lbrace \begin{array}{llll}
%\d X^{1}_t &= &F_{1}(t,X^n_t)\d t + \sigma(t,X_t^{n}) \d B_t ,\qquad & X^1_0=x_1,\\
%\d X^2_t &= &X_t^1 \d t,\qquad & X^2_0=x_2,\\
%\vdots\\
%\d X_t^{n-1} &= &X_{t}^{n-2} \d t,\qquad & X^{n-1}_0=x_{n-1},\\
%\d X^{n}_t &= & F_n(t,X^{n-1}_t,X^n_t)\d t,\qquad & X^n_0=x_n.
%\end{array}
%\right.
%\end{equation}
%C'est à dire la perturbation de la dérive d'une équation par la $(n-1)$-ième intégrale d'une diffusion. On a alors un seuil de régularité de minimal de $(n-3/2)/(n-1/2)$ et on peut encore écrire le contre exemple. 
%
%
%

\bibliographystyle{amsalpha}
\bibliography{Biblio}

\providecommand{\bysame}{\leavevmode\hbox to3em{\hrulefill}\thinspace}
\providecommand{\MR}{\relax\ifhmode\unskip\space\fi MR }
\providecommand{\MRhref}[2]{%
  \href{http://www.ams.org/mathscinet-getitem?mr=#1}{#2}
}
\providecommand{\href}[2]{#2}
\begin{thebibliography}{BFGM14}

\bibitem[BFGM14]{beck_stochastic_2014}
Lisa Beck, Franco Flandoli, Massimiliano Gubinelli, and Mario Maurelli,
  \emph{Stochastic {ODEs} and stochastic linear {PDEs} with critical drift:
  regularity, duality and uniqueness}, arXiv:1401.1530 [math] (2014).

\bibitem[CdR12]{chaudru_strong_2012}
Paul-Eric Chaudru~de Raynal, \emph{Strong existence and uniqueness for
  stochastic differential equation with hölder drift and degenerate noise},
  Annales de l'Institut Henri Poincaré, Probabilités et Statistique (2012),
  To appear.

\bibitem[CG12]{catellier_averaging_2012}
R.~Catellier and M.~Gubinelli, \emph{Averaging along irregular curves and
  regularisation of {ODEs}}, arXiv:1205.1735 [math] (2012), arXiv: 1205.1735.

\bibitem[DD15]{delarue_rough_2015}
François Delarue and Roland Diel, \emph{Rough paths and 1d {SDE} with a time
  dependent distributional drift: application to polymers}, Probability Theory
  and Related Fields (2015), 1--63 (en).

\bibitem[DF14]{delarue_transition_2014}
François Delarue and Franco Flandoli, \emph{The transition point in the zero
  noise limit for a 1d {Peano} example}, Discrete and Continuous Dynamical
  Systems \textbf{34} (2014), no.~10, 4071--4083 (en).

\bibitem[DFP06]{di_francesco_schauder_2006}
Marco Di~Francesco and Sergio Polidoro, \emph{Schauder estimates, {Harnack}
  inequality and {Gaussian} lower bound for {Kolmogorov}-type operators in
  non-divergence form}, Advances in Differential Equations \textbf{11} (2006),
  no.~11, 1261--1320.

\bibitem[DL89]{diperna_ordinary_1989}
R.~J. DiPerna and P.-L. Lions, \emph{Ordinary differential equations, transport
  theory and {Sobolev} spaces}, Inventiones Mathematicae \textbf{98} (1989),
  no.~3, 511--547.

\bibitem[DM10]{delarue_density_2010}
François Delarue and Stéphane Menozzi, \emph{Density estimates for a random
  noise propagating through a chain of differential equations}, Journal of
  Functional Analysis \textbf{259} (2010), no.~6, 1577--1630.

\bibitem[FFPV16]{fedrizzi_regularity_2016}
Ennio Fedrizzi, Franco Flandoli, Enrico Priola, and Julien Vovelle,
  \emph{Regularity of {Stochastic} {Kinetic} {Equations}}, arXiv:1606.01088
  [math] (2016), arXiv: 1606.01088.

\bibitem[FIR14]{flandoli_multidimensional_2014}
Franco Flandoli, Elena Issoglio, and Francesco Russo, \emph{Multidimensional
  stochastic differential equations with distributional drift}, arXiv:1401.6010
  [math] (2014).

\bibitem[Fla11]{flandoli_random_2011}
Franco Flandoli, \emph{Random perturbation of {PDEs} and fluid dynamic models},
  Lecture {Notes} in {Mathematics}, vol. 2015, Springer, Heidelberg, 2011,
  Lectures from the 40th Probability Summer School held in Saint-Flour, 2010.

\bibitem[Fri64]{friedman_partial_1964}
Avner Friedman, \emph{Partial differential equations of parabolic type},
  Prentice-Hall Inc., Englewood Cliffs, N.J., 1964.

\bibitem[Hai15]{hairer_introduction_2015}
Martin Hairer, \emph{Introduction to regularity structures}, Brazilian Journal
  of Probability and Statistics \textbf{29} (2015), no.~2, 175--210 (EN).
  \MR{MR3336866}

\bibitem[Hö67]{hormander_hypoelliptic_1967}
Lars Hörmander, \emph{Hypoelliptic second order differential equations}, Acta
  Mathematica \textbf{119} (1967), 147--171.

\bibitem[Kol34]{kolmogorov_zufallige_1934}
A~Kolmogorov, \emph{Zufällige bewegungen. (zur theorie der brownschen
  bewegung.).}, Ann. of Math., II. Ser. \textbf{35} (1934), 116--117.

\bibitem[KR05]{krylov_strong_2005}
Nicolai~Vladimirovitch Krylov and Michael Röckner, \emph{Strong solutions of
  stochastic equations with singular time dependent drift}, Probability Theory
  and Related Fields \textbf{131} (2005), no.~2, 154--196 (en).

\bibitem[Men11]{menozzi_parametrix_2011}
Stéphane Menozzi, \emph{Parametrix techniques and martingale problems for some
  degenerate {Kolmogorov} equations}, Electronic Communications in Probability
  \textbf{16} (2011), 234--250.

\bibitem[MS67]{mckean_jr._curvature_1967}
H.~P. McKean, Jr. and I.~M. Singer, \emph{Curvature and the eigenvalues of the
  {Laplacian}}, Journal of Differential Geometry \textbf{1} (1967), no.~1,
  43--69.

\bibitem[SV79]{stroock_multidimensional_1979}
Daniel~W. Stroock and S.~R.~Srinivasa Varadhan, \emph{Multidimensional
  diffusion processes}, Grundlehren der {Mathematischen} {Wissenschaften}
  [{Fundamental} {Principles} of {Mathematical} {Sciences}], vol. 233,
  Springer-Verlag, Berlin, 1979.

\bibitem[Ver80]{veretennikov_strong_1980}
Alexander~Ju. Veretennikov, \emph{Strong solutions and explicit formulas for
  solutions of stochastic integral equations},
  Matematicheski{\textbackslash}uı{\textbackslash} Sbornik. Novaya Seriya
  \textbf{111(153)} (1980), no.~3, 434--452, 480.

\bibitem[WZ15]{wang_degenerate_2015}
Feng-Yu Wang and Xicheng Zhang, \emph{Degenerate {SDE} with {H}ölder-{Dini}
  {Drift} and {Non}-{Lipschitz} {Noise} {Coefficient}}, arXiv:1504.04450 [math]
  (2015), arXiv: 1504.04450.

\bibitem[Zha05]{zhang_strong_2005}
Xicheng Zhang, \emph{Strong solutions of {SDES} with singular drift and
  {Sobolev} diffusion coefficients}, Stochastic Processes and their
  Applications \textbf{115} (2005), no.~11, 1805--1818.

\bibitem[Zvo74]{zvonkin_transformation_1974}
A.~K. Zvonkin, \emph{A transformation of the phase space of a diffusion process
  that will remove the drift}, Mat. Sb. (N.S.) \textbf{93(135)} (1974),
  129--149, 152.

\end{thebibliography}

\end{document}